\documentclass[11pt]{article}

\usepackage[T1]{fontenc}
\usepackage[utf8]{inputenc}
\usepackage{lmodern}
\usepackage{amsmath,amssymb,amsthm}
\usepackage[a4paper,margin=29mm]{geometry}
\usepackage[hidelinks]{hyperref}

\newtheorem{theorem}{Theorem}[section]
\newtheorem{proposition}[theorem]{Proposition}
\newtheorem{lemma}[theorem]{Lemma}
\newtheorem{corollary}[theorem]{Corollary}
\theoremstyle{remark}

\newcommand{\D}{\mathbb D}
\newcommand{\T}{\mathbb T}
\newcommand{\C}{\mathbb C}
\newcommand{\W}{W}

\title{Sch\"affer's matrix inequality: the exact asymptotic constant}
\author{%
Samy Houache\textsuperscript{1}\quad
Oleg Szehr\textsuperscript{2}\quad
Rachid Zarouf\textsuperscript{3,4}\\[0.6em]
\footnotesize \textsuperscript{1}University of Bordeaux, Institut de Math\'ematiques de Bordeaux, Thales AVS, France\\
\footnotesize \textsuperscript{2}Dalle Molle Institute for Artificial Intelligence (IDSIA),\\[-0.15em]
\footnotesize SUPSI--USI, Lugano-Viganello, Switzerland\\
\footnotesize \textsuperscript{3}Aix-Marseille University, University of Toulon, CNRS, CPT, Marseille, France\\
\footnotesize \textsuperscript{4}Aix-Marseille University, Laboratory ADEF, Marseille, France\\[0.35em]
\footnotesize E-mail: \href{mailto:samy.houache@math.u-bordeaux.fr}{\texttt{samy.houache@math.u-bordeaux.fr}}\\[-0.15em]
\footnotesize E-mail: \href{mailto:oleg.szehr@idsia.ch}{\texttt{oleg.szehr@idsia.ch}}\\[-0.15em]
\footnotesize E-mail: \href{mailto:rachid.zarouf@univ-amu.fr}{\texttt{rachid.zarouf@univ-amu.fr}}}
\date{}

\begin{document}
\maketitle
\begin{abstract}
Let $\mathcal S_n$ denote the smallest constant such that
\[
|\det T|\|T^{-1}\|
\leq \mathcal S_n \|T\|^{n-1}
\]
for every invertible operator $T$ on every $n$-dimensional complex
Banach space. In Hilbert space the optimal constant is $1$. For
arbitrary Banach spaces, J.~J.~Sch\"affer proved in 1970 that
\[
\mathcal S_n\leq \sqrt{en}.
\]
Subsequent work showed that $\mathcal S_n$ grows like $\sqrt n$, but the
sharp asymptotic constant has remained open for more than five decades.
We resolve this problem by proving
\[
\lim_{n\to\infty}\frac{\mathcal S_n}{\sqrt n}=\sqrt e.
\]
Thus Sch\"affer's upper bound is asymptotically sharp, including its constant.

Our proof is constructive, providing explicit Banach-space norms through duality and explicit matrices through the theory of model operators.
At the analytic core of the argument, an extremal formulation of Sch\"affer's problem in the Wiener algebra reduces the matching asymptotic lower bound for $\mathcal S_n$ to uniformly controlling the Taylor coefficients of products $QB_n$, where $B_n$ is a finite Blaschke product of degree $n$ and $Q$ is a polynomial factor. We optimize
simultaneously the zero distribution of $B_n$ and the choice of $Q$.
The resulting zeros follow a logarithmic asymptotic distribution, and a
sharp uniform asymptotic analysis of the Taylor coefficients of $QB_n$ yields the
constant $\sqrt e$. The corresponding model operators then yield
matrices with these spectra that asymptotically attain Sch\"affer's bound.
\end{abstract}



\section{Introduction}\label{sec:introduction}

A basic problem in finite-dimensional operator theory is to control the
norm of the inverse of an operator from information about the operator
itself. A particularly natural scale-invariant form of this problem
combines the norm of $T$, the norm of its inverse, and its determinant, which measures the global volume scaling of $T$.

In the Hilbert space case the relation is exact. Indeed, if
$s_1(T)\geq\cdots\geq s_n(T)>0$ are the singular values of an
invertible operator $T$, then
\[
 |\det T|\,\|T^{-1}\|_2
 =\prod_{j=1}^{n-1}s_j(T)
 \leq \|T\|_2^{n-1}.
\]
Here $\C^n$ is equipped with its Euclidean norm and $\|\cdot\|_2$
denotes the induced operator norm. The constant $1$ is attained by the
identity. Thus, in Hilbert space, volume scaling and extremal stretching
control the size of the inverse with no loss, uniformly in the dimension.

The situation changes when the Hilbertian structure is removed. For an
arbitrary norm on $\C^n$, we denote by $\|\cdot\|$ the corresponding
induced operator norm. The determinant still measures volume scaling,
while $\|T\|$ and $\|T^{-1}\|^{-1}$ measure, respectively, the
maximal expansion and minimal contraction of $T$. Unlike in the Hilbert
space case, however, there is no singular-value decomposition compatible
with an arbitrary norm that relates these quantities.

The corresponding universal inverse problem is therefore to determine
the smallest constant $\mathcal S_n$ for which
\[
 |\det T|\,\|T^{-1}\|
 \leq \mathcal S_n\,\|T\|^{n-1}
\]
holds for every norm on $\C^n$ and every invertible operator
$T\in M_n(\C)$. From this viewpoint, $\mathcal S_n$ measures the worst
possible loss in the Hilbert-space inverse inequality caused by
arbitrary Banach-space geometry.

This problem was introduced systematically by J.~J.~Sch\"affer in his
1970 paper \emph{Norms and determinants of linear mappings}
\cite{Schaeffer}. He proved the universal estimate
\[
\mathcal S_n\leq\sqrt{en},
\]
and conjectured that, as in the Hilbert space case, the sequence
$(\mathcal S_n)_n$ should remain bounded. This conjecture was disproved
by E.~Gluskin, M.~Meyer and A.~Pajor~\cite{GMP}. Subsequent lower bounds, including a probabilistic argument of J.~Bourgain\footnote{The same article~\cite{GMP} contains, as Theorem~5, a stronger result due to Bourgain.}~\cite[Theorem~5]{GMP} connected Sch\"affer's problem to Tur\'an-type power-sum estimates. Building on Bourgain's estimate from~\cite{GMP} and a number-theoretic construction of H.~Montgomery \cite[Example~6, p.~101]{Montgomery}, also discussed by P.~Tur\'an in \cite[p.~83]{Turan84}, H.~Queff\'elec~\cite{Queffelec93} established that $\mathcal S_n$ has the correct order of growth $\sqrt n$.
These arguments are essentially nonconstructive: probabilistic or
number-theoretic estimates for power sums establish the required lower
bounds without producing explicit matrices and Banach-space norms
realizing them. As a later formulation of the remaining problem, J.~Andersson~\cite{Andersson06} conjectured in 2006 that $\mathcal S_n/\sqrt n$ converges and asked for the value of its limit.

A different, constructive approach that avoids the power-sum method was
introduced by O.~Szehr and R.~Zarouf~\cite{SZJMPA}. Using
function-theoretic duality together with the theory of model operators,
they obtained, respectively, explicit Banach-space norms and explicit
matrices exhibiting the order $\sqrt n$. The present paper follows this
constructive line of argument and develops it further to determine the
sharp asymptotic constant. More precisely, we prove the following theorem; the precise construction
of the norms and matrices is given in
Theorem~\ref{thm:explicit-matrices} below.
\begin{theorem}\label{thm:main-intro}
For each $n\geq2$, there exist a norm $\|\cdot\|_*$ on $\C^n$ and an
invertible matrix $T_n\in M_n(\C)$, both given explicitly, such that
\[
 \frac{|\det T_n|\,\|T_n^{-1}\|_*}
      {\|T_n\|_*^{\,n-1}\sqrt n}
 \longrightarrow \sqrt e .
\]
Consequently,
\[
 \lim_{n\to\infty}\frac{\mathcal S_n}{\sqrt n}=\sqrt e.
\]
In particular, Sch\"affer's upper bound
$\mathcal S_n\leq\sqrt{en}$ is asymptotically sharp.
\end{theorem}

This identifies the precise first-order dimension-dependent loss in the inverse inequality caused by arbitrary Banach-space geometry: the optimal universal loss relative to the Hilbert-space inequality is
$(\sqrt e+o(1))\sqrt n$. 

The analytic mechanism underlying the construction is the H\"older
duality between the Wiener algebra $\W$ and $l_\infty^A$. The
Gluskin--Meyer--Pajor representation expresses $\mathcal S_n$ as the
supremum, over all spectra $(\lambda_1,\ldots,\lambda_n)\in\D^n$, of a Wiener-algebra interpolation functional $\varphi(\lambda_1,\ldots,\lambda_n)$; see
Section~\ref{sec:duality}. For a prescribed spectrum, let $B$ be the
corresponding finite Blaschke product. If $h$ is admissible for the
interpolation problem, the division property of $\W$ gives $h/B\in\W$.
Pairing $h$ with $QB$, where $Q$ is a suitable monic polynomial, and
applying H\"older duality yields
\[
 \varphi(\lambda_1,\ldots,\lambda_n)
 \geq
 \frac{1}{\|QB\|_{l_\infty^A}}
 -\prod_{j=1}^n|\lambda_j|.
\]
Thus, for a fixed spectrum, the required lower estimate reduces to
controlling the largest Taylor coefficient of $QB$. In
\cite{SZJMPA}, the choice of a fixed singleton spectrum and
$Q=1-z^2$ yields the order $\sqrt n$.
Here we allow the spectrum to vary with $n$ and optimize both its asymptotic distribution and the polynomial factor. This leads to a logarithmic asymptotic distribution of the zeros of $B_n$. After discretizing this distribution, a sharp uniform asymptotic
analysis of the Taylor coefficients of $QB_n$ produces the constant
$\sqrt e$.

\subsubsection*{Outline of the paper.}
Section~\ref{sec:duality} recalls the Gluskin--Meyer--Pajor
representation~\cite{GMP} of $\mathcal S_n$ and the sharp
H\"older-duality estimate from \cite{SZJMPA}.
Section~\ref{sec:spectrum-polynomial} constructs the sharp lower bound.
In Subsection~\ref{sec:limiting-extremal}, stationary-phase
considerations lead to a limiting extremal problem for the distribution
of the zeros of the Blaschke product, whose optimal density is the
logarithmic law $\tau(t)=\log\frac{2\pi}{t}$.
In Subsection~\ref{sec:polynomial-factor}, we optimize the monic
polynomial factor $Q$ for the resulting weight and show that the sharp
value of the associated weighted polynomial extremal problem is
$e^{-1/2}$. Finally, in Subsection~\ref{sec:eigenvalue-distribution},
we discretize the logarithmic density to obtain the explicit spectra
$(\lambda_{j,n})_{j=1}^n$ and we establish the uniform asymptotic estimate
for the largest Taylor coefficient of $QB_n$ in Proposition~\ref{prop:coeff-norm}. Combining this estimate
with the model-operator construction yields our main result,
Theorem~\ref{thm:explicit-matrices}. The proof of Proposition~\ref{prop:coeff-norm} is deferred to
Appendix~\ref{sec:coeff-proof}; it follows the same asymptotic
approach as in \cite{SZJMPA}, combining stationary phase with van der
Corput estimates.

\section{Duality and interpolation for a prescribed spectrum}\label{sec:duality}

For $f\in\operatorname{Hol}(\D)$ write
\[
 f(z)=\sum_{k\ge0}\widehat f(k)z^k,
 \qquad
 \widehat f(k)=\frac{f^{(k)}(0)}{k!}.
\]
We use the two coefficient norms
\[
 \|f\|_{\W}=\sum_{k\ge0}|\widehat f(k)|,
 \qquad
 \|f\|_{{l_\infty^A}}=\sup_{k\ge0}|\widehat f(k)|,
\]
and denote by $\W$ and $l_\infty^A$ the corresponding spaces of
functions in $\operatorname{Hol}(\D)$ for which these norms are finite.
For $f\in\W$ and $g\in l_\infty^A$, define
\[
 \langle f,g\rangle
 =\sum_{k\ge0}\widehat f(k)\overline{\widehat g(k)}.
\]
The series is absolutely convergent, and H\"older's inequality gives
\[
 |\langle f,g\rangle|\le \|f\|_{\W}\,\|g\|_{{l_\infty^A}}.
\]
When $f,g\in H^2$ this bracket is simply the normalized
$L^2(\T)$ inner product of their boundary values:
\[
 \langle f,g\rangle
 =\frac1{2\pi}\int_0^{2\pi}
 f(e^{it})\overline{g(e^{it})}\,dt.
\]

Let $\lambda_1,\ldots,\lambda_n$ be distinct points of
$\D\setminus\{0\}$. Put
\[
 d=\prod_{j=1}^n\lambda_j,
 \qquad
 B(z)=\prod_{j=1}^n b_{\lambda_j}(z),
 \qquad
 b_\lambda(z)=\frac{z-\lambda}{1-\overline\lambda z}.
\]
Thus $B(0)=(-1)^nd$. For such a spectrum, define, as in
\cite{SZJMPA},
\begin{equation}\label{eq:varphi}
 \varphi(\lambda_1,\ldots,\lambda_n)=
 \inf\left\{\|h\|_{\W}-|h(0)|:
 h\in\W,\ h(0)=d,\ h(\lambda_j)=0,\ 1\le j\le n\right\}.
\end{equation}

The significance of $\varphi$ comes from the
Gluskin--Meyer--Pajor representation~\cite{GMP}: the optimization over
all norms and invertible matrices in Sch\"affer's problem can be reduced
exactly to an optimization over spectra,
\[
 \mathcal S_n
 =
 \sup_{(\lambda_1,\ldots,\lambda_n)\in\D^n}
 \varphi(\lambda_1,\ldots,\lambda_n),
\]
with the interpolation conditions understood with multiplicity in the
general case. In particular, every prescribed spectrum
$(\lambda_1,\ldots,\lambda_n)$ gives the lower bound
\[
 \mathcal S_n\geq
 \varphi(\lambda_1,\ldots,\lambda_n).
\]
Thus lower bounds for $\mathcal S_n$ may be obtained by prescribing a
spectrum and estimating its interpolation value $\varphi$.

The corresponding prescribed-spectrum extremal problem can moreover be
realized by model operators. This belongs to the general model-operator
approach to interpolation in function algebras; see, for instance,
N.~K.~Nikolski~\cite{Nikolski05}. In the Wiener-algebra setting
relevant here, multiplication by $z$ on the quotient $\W/B\W$,
together with a functional-calculus argument for $1/z$, realizes
$\varphi$ by an explicit matrix; see
\cite[Lemma~7, Theorem~8 and Remark~12]{SZJMPA}. Consequently, an
effective lower bound for $\varphi$ not only yields a lower bound for
$\mathcal S_n$, but also leads to an explicit matrix example.

The following estimate is the H\"older-duality estimate used in Section~3 of \cite{SZJMPA}, where the test polynomial was $1-z^2$ and the Blaschke product was simply $B(z)=b_\lambda(z)^n$. We state it with an arbitrary monic polynomial because this additional freedom is what will be optimized below. As explained in \cite[Remark~10]{SZJMPA}, the underlying duality approach is sharp.

\begin{proposition}\label{prop:holder-lower}
For every monic polynomial $Q$,
\begin{equation}\label{eq:holder-lower}
 \varphi(\lambda_1,\ldots,\lambda_n)\ge
 \frac1{\|QB\|_{{l_\infty^A}}}-|d|.
\end{equation}
\end{proposition}

\begin{proof}
Let $h$ be admissible in \eqref{eq:varphi}. By the division property of the Wiener algebra \cite[Section~3]{SZJMPA}, $h=Bu$ for some $u\in\W$, and $u(0)=(-1)^n$. If $m=\deg Q$, then $Q$ is monic and $|B|=1$ on $\T$, hence
\[
 \langle z^mh,QB\rangle=\langle z^mu,Q\rangle=(-1)^n.
\]
Since $\|z^mh\|_{\W}=\|h\|_{\W}$, H\"older's inequality gives
\[
 1\le \|h\|_{\W}\,\|QB\|_{{l_\infty^A}}.
\]
Taking the infimum over $h$ proves \eqref{eq:holder-lower}.
\end{proof}

\section{Construction of the sharp lower bound}\label{sec:spectrum-polynomial}

\subsection{Optimization of the distribution of zeros}\label{sec:limiting-extremal}

Let
\[
 \lambda_{j,n}=r_ne^{i\theta_{j,n}},\qquad 1\leq j\leq n,
\]
be distinct points. The radial parameters $0<r_n<1$ will be chosen so
that
\[
 r_n\to1,\qquad r_n^n\to0,
\]
while the angles $\theta_{j,n}$ will be specified later. Since
$|d|=r_n^n$, the second condition makes the determinant term in
Proposition~\ref{prop:holder-lower} asymptotically negligible. Consider the Blaschke product and the empirical angular distribution
\[
 B_n(z)=\prod_{j=1}^n b_{\lambda_{j,n}}(z),
 \qquad
 \mu_n=\frac1n\sum_{j=1}^n\delta_{\theta_{j,n}}.
\]
Writing
\[
 B_n(e^{it})=e^{i\phi_n(t)},
 \qquad
 \tau_n(t)=\frac{\phi_n'(t)}n,
\]
we have
\begin{equation}\label{eq:tau-poisson-measure}
 \tau_n(t)
 =\int_0^{2\pi}P_{r_n}(t-s)\,d\mu_n(s),
 \qquad
 P_\rho(x)=\frac{1-\rho^2}{1-2\rho\cos x+\rho^2},
\end{equation}
where $P_\rho$ is the Poisson kernel. Thus $\tau_n$ is the Poisson regularization of the
distribution $\mu_n$. Since
\[
 \frac1{2\pi}\int_0^{2\pi}\tau_n(t)\,dt=1,
\]
we seek limiting profile of the functions $\tau_n$ that satisfy the same
normalization.

We now show, by stationary-phase considerations, how such a limiting
profile $\tau$ governs the asymptotic size of the Taylor coefficients
of $QB_n$, and thereby derive the extremal problem for the distribution
of the zeros. The uniform
estimates needed to justify this passage for the discrete construction
are proved in Proposition~\ref{prop:coeff-norm} and
Appendix~\ref{sec:coeff-proof}. For a Fourier index $k\ge0$, put $a=k/n$. Then
\[
 \widehat{QB_n}(k)
 =\frac1{2\pi}\int_0^{2\pi}
 Q(e^{it})e^{in\psi_{n,a}(t)}\,dt,
 \qquad
 \psi_{n,a}(t)=\frac{\phi_n(t)}n-at,
\]
and hence
\begin{equation}\label{eq:phase-derivatives-motivation}
 \psi_{n,a}'(t)=\tau_n(t)-a,
 \qquad
 \psi_{n,a}''(t)=\tau_n'(t).
\end{equation}
Thus a stationary point $t_{n,a}$ satisfies $\tau_n(t_{n,a})=a$. Suppose that $\tau_n\to\tau$ and $\tau_n'\to\tau'$ locally uniformly on
$(0,2\pi)$, where $\tau>0$ and $\tau'<0$. If
\[
 a_n=\frac{k_n}{n}\longrightarrow a
\]
and the corresponding stationary points remain in a compact subinterval of
$(0,2\pi)$, then
\[
 t_{n,a_n}\longrightarrow t_a,\qquad
 \tau(t_a)=a,\qquad
 \tau_n'(t_{n,a_n})\longrightarrow\tau'(t_a)<0.
\]
Under these assumptions, stationary phase gives
\[
 \sqrt n\,|\widehat{QB_n}(k_n)|
 =
 \frac{|Q(e^{it_{n,a_n}})|}
      {\sqrt{2\pi[-\tau_n'(t_{n,a_n})]}}+o(1),
\]
and passing to the limit in the prefactor gives
\[
 \sqrt n\,|\widehat{QB_n}(k_n)|
 =
 \frac{|Q(e^{it_a})|}
      {\sqrt{2\pi[-\tau'(t_a)]}}+o(1).
\]
This suggests introducing the weight
\[
 \omega_\tau(t)
 =\frac1{\sqrt{2\pi[-\tau'(t)]}},
 \qquad 0<t<2\pi.
\]
This suggests that, at the limiting level,
$\sqrt n\,\|QB_n\|_{l_\infty^A}$ is governed by the weighted maximum
\[
 \sup_{0<t<2\pi}\omega_\tau(t)|Q(e^{it})|.
\]
We are therefore led, within the class of positive decreasing densities, to the continuum extremal problem
\[
 \inf_{\tau,Q}
 \sup_{0<t<2\pi}\omega_\tau(t)|Q(e^{it})|,
\]
where $\tau$ ranges over positive decreasing probability densities
with respect to $dt/(2\pi)$ and $Q$ over monic polynomials. We first optimize the density by deriving a
lower bound independent of $Q$; the polynomial factor is optimized in
the next subsection. Indeed, Jensen's formula gives, for every monic $Q$,
\[
 \frac1{2\pi}\int_0^{2\pi}\log|Q(e^{it})|\,dt\ge0.
\]
Hence
\begin{align*}
 \log\sup_t {\omega_\tau}(t)|Q(e^{it})|
 &\ge \frac1{2\pi}\int_0^{2\pi}
       \log\bigl({\omega_\tau}(t)|Q(e^{it})|\bigr)\,dt\\
 &\ge \frac1{2\pi}\int_0^{2\pi}\log {\omega_\tau}(t)\,dt.
\end{align*}
Thus the logarithmic mean of $\omega_\tau$ gives a lower bound for
the logarithm of the weighted maximum, uniformly over all monic
polynomials $Q$. The following proposition shows that this logarithmic
mean is at least $-1/2$ and identifies the unique equality case.

\begin{proposition}\label{prop:log-law}
Assume that $\tau\in C^1((0,2\pi])$ satisfies
\[
 \tau(t)>0,\qquad \tau'(t)<0,\qquad0<t<2\pi,
\]
\[
 \frac1{2\pi}\int_0^{2\pi}\tau(t)\,dt=1,
\]
and that $\log(-\tau')\in L^1(0,2\pi)$. Then
\[
 \frac1{2\pi}\int_0^{2\pi}\log {\omega_\tau}(t)\,dt\ge-\frac12.
\]
Equality holds only for
\[
 \tau(t)=\log\frac{2\pi}{t}.
\]
\end{proposition}

\begin{proof}
Put $p=-\tau'>0$ and $c=\tau(2\pi)$. Continuity at $2\pi$ and positivity on
$(0,2\pi)$ give $c\ge0$, while strict decrease and the normalization imply
$c<1$: if $c\ge1$, then $\tau(t)>1$ for $0<t<2\pi$, contrary to the
normalization. Moreover, positivity, monotonicity and integrability imply
\[
 t\tau(t)\longrightarrow0,\qquad t{\to}0,
\]
because
\[
 \frac t2\tau(t)
 \le \int_{t/2}^{t}\tau(s)\,ds
 \le \int_0^t\tau(s)\,ds\longrightarrow0.
\]
We may therefore integrate by parts on $[\varepsilon,2\pi]$ and let
$\varepsilon{\to}0$, obtaining
\begin{align*}
 \frac1{2\pi}\int_0^{2\pi}t p(t)\,dt
 &=\frac1{2\pi}\left([-t\tau(t)]_0^{2\pi}
          +\int_0^{2\pi}\tau(t)\,dt\right)\\
 &=1-c.
\end{align*}
Jensen's inequality applied to the positive function $tp(t)$
gives
\[
 \frac1{2\pi}\int_0^{2\pi}\log(tp(t))\,dt\le\log(1-c).
\]
Also,
\[
 \frac1{2\pi}\int_0^{2\pi}\log t\,dt=\log(2\pi)-1.
\]
Since
\[
 \log {\omega_\tau}(t)=-\frac12\log(2\pi)-\frac12\log p(t),
\]
we obtain
\[
 \frac1{2\pi}\int_0^{2\pi}\log {\omega_\tau}(t)\,dt
 \ge-\frac12-\frac12\log(1-c)\ge-\frac12.
\]
For equality, the second inequality first forces $c=0$. Equality in Jensen
then forces $tp(t)$ to be constant. Its mean is $1$, so $tp(t)=1$ and
$p(t)=1/t$. Since $p=-\tau'$ and $\tau(2\pi)=0$, this gives
\[
 \tau(t)=\int_t^{2\pi}\frac{du}{u}=\log\frac{2\pi}{t}.
\]
This function satisfies all the assumptions and gives equality.
\end{proof}

Combining the proposition with the preceding inequality gives
\[
 \sup_{0<t<2\pi}\omega_\tau(t)|Q(e^{it})|
 \geq e^{-1/2}
\]
for every admissible density $\tau$ and every monic polynomial $Q$.
Moreover, the density-dependent lower bound is sharp only for $\tau(t)=\log\frac{2\pi}{t}$.
For this density, the corresponding probability measure and weight are
\[
 d\nu(t)=\log\frac{2\pi}{t}\,\frac{dt}{2\pi},
 \qquad
 \omega_\tau(t)=\sqrt{\frac{t}{2\pi}}.
\]
It remains to show that, for this weight, the lower bound $e^{-1/2}$
can be approached by a suitable monic polynomial $Q$. This is the
subject of the next subsection.

\subsection{Optimization of the polynomial factor}
\label{sec:polynomial-factor}

For the logarithmic density found in
Proposition~\ref{prop:log-law}, the corresponding weight is $\omega_\tau(t)=\sqrt{\frac{t}{2\pi}}$.
We now optimize the monic polynomial factor $Q$ for this weight. To
control the contributions near $t=0$ and $t=2\pi$ in the subsequent
uniform coefficient estimate, we require $Q$ to have a zero of fixed
order at $1$. For a polynomial $Q$, set
\[
 C(Q)=\max_{0\leq t\leq2\pi}
 \sqrt{\frac{t}{2\pi}}\,|Q(e^{it})|.
\]
The following theorem determines the sharp value of $C(Q)$ over this
class of monic polynomials and shows that imposing the prescribed
vanishing at $1$ does not alter the extremal value. This extremal problem is a weighted Chebyshev problem on the unit circle. The result in Theorem~\ref{thm:Q} also follows from H.~Widom's general asymptotic formula~\cite[Theorem~8.3]{Widom69}. More generally, Widom studies, for each fixed degree, the minimum of $\max_E \rho|P|$ over monic polynomials $P$, where $E$ may be a finite union of sufficiently smooth Jordan curves; here $E=\mathbb T$. We give a short elementary proof adapted to this case.

\begin{theorem}\label{thm:Q}
For every integer $L\ge1$,
\[
 \inf_{\substack{Q\ \mathrm{monic}\\(z-1)^L\mid Q}} C(Q)
 =e^{-1/2}.
\]
\end{theorem}

\begin{proof}
\smallskip\noindent\emph{Lower bound.}
For every monic polynomial $Q$, Jensen's formula gives
\[
 \frac1{2\pi}\int_0^{2\pi}\log|Q(e^{it})|\,dt\geq0.
\]
Since
\[
 \frac1{2\pi}\int_0^{2\pi}
 \log\sqrt{\frac{t}{2\pi}}\,dt=-\frac12,
\]
we obtain
\[
 \log C(Q)
 \geq
 \frac1{2\pi}\int_0^{2\pi}
 \log\!\left(\sqrt{\frac{t}{2\pi}}\,|Q(e^{it})|\right)\,dt
 \geq-\frac12
\]
and consequently $C(Q)\geq e^{-1/2}$.

\smallskip\noindent\emph{Upper bound: Construction of almost optimal polynomials.}
For the reverse inequality, fix $\Gamma>e^{-1/2}$. We construct a monic
polynomial $Q$, divisible by $(z-1)^L$, such that $C(Q)<\Gamma$. Put
\[
 {\Omega_L(t)=\sqrt{\frac{t}{2\pi}}\,|e^{it}-1|^L}.
\]
The logarithmic mean of $\Omega_L$ is $-1/2$. Indeed, the first factor has mean
$-1/2$, while
\[
 \frac1{2\pi}\int_0^{2\pi}\log|e^{it}-1|\,dt=0.
\]
Hence, for $0<t<2\pi$,
\[
 v(t)=\log\frac{\Gamma}{\Omega_L(t)}
\]
has mean
\[
 \frac1{2\pi}\int_0^{2\pi}v(t)\,dt
 =\log\Gamma+\frac12>0.
\]
The positivity of the mean of $v$ is the key point of the construction.
We seek a polynomial $q$ with $q(0)=1$ such that
\[
 \log|q(e^{it})|<v(t),\qquad 0<t<2\pi,
\]
for then
\[
 \Omega_L(t)|q(e^{it})|<\Gamma.
\]
We construct such a polynomial in three steps. First, we find an
analytic polynomial $H$ with $H(0)=0$ whose real part lies strictly
below $v$ on $\T$. We then approximate $e^H$ by polynomials $q_m$
satisfying $q_m(0)=1$. Finally, we reverse $q_m$ to obtain the required
monic polynomial.

\smallskip\noindent\emph{Step 1: Construction of an analytic logarithmic minorant.}
Since $v(t)\to+\infty$ at both endpoints, for $M>0$ define
\[
 v_M(t)=\min(v(t),M),\qquad 0<t<2\pi,
\]
and set $v_M(0)=v_M(2\pi)=M$. Then $v_M$ is continuous and $2\pi$-periodic.
Moreover, $v_M(t)\to v(t)$ for almost every $t$ as $M\to\infty$.
Since $v$ has only logarithmic endpoint singularities, $v\in L^1(0,2\pi)$,
and dominated convergence gives
\[
 \frac1{2\pi}\int_0^{2\pi}v_M(t)\,dt
 \longrightarrow
 \frac1{2\pi}\int_0^{2\pi}v(t)\,dt>0.
\]
Hence, for $M$ large enough,
\[
 a_0=\frac1{2\pi}\int_0^{2\pi}v_M(t)\,dt>0.
\]

Put
\[
 u_0=v_M-a_0.
\]
Then $u_0$ is continuous, periodic, and has mean zero. Its Fej\'er sums
converge uniformly to $u_0$ and have the same mean. We may therefore choose a
real trigonometric polynomial $u$ of mean zero such that
\[
 \|u-u_0\|_\infty<\frac{a_0}{4}.
\]
Since $v_M\le v$, this gives
\begin{equation}\label{eq:u-below-v}
 u(t)\le v(t)-\frac{3a_0}{4}.
\end{equation}
Because $u$ is real and has mean zero, it is the real part of an analytic
polynomial with zero constant term. Write
\[
 u(t)=2\operatorname{Re}\sum_{j=1}^N c_je^{ijt},
 \qquad
 H(z)=2\sum_{j=1}^Nc_jz^j.
\]
Then $H(0)=0$ and
\[
 \operatorname{Re}H(e^{it})=u(t).
\]

\smallskip\noindent\emph{Step 2: Polynomial approximation of the exponential.}
The analytic function $e^H$ has value $1$ at the origin and, by
\eqref{eq:u-below-v},
\[
 \log|e^{H(e^{it})}|=u(t)
 \leq v(t)-\frac{3a_0}{4}.
\]
Thus $e^H$ has the desired boundary modulus, but is not in general a
polynomial. We approximate it by setting, for a positive integer $m$,
\[
 q_m(z)=\left(1+\frac{H(z)}{m}\right)^m.
\]
Then $q_m$ is a polynomial and, since $H(0)=0$, $q_m(0)=1$. Let
\[
 M_H=\max_{|z|\leq1}|H(z)|.
\]
For $m>2M_H$, the Taylor expansion of the principal logarithm gives,
with a numerical constant $C$,
\[
 \biggl|m\log\left(1+\frac{H(z)}m\right)-H(z)\biggr|
 \leq \frac{CM_H^2}{m},
 \qquad |z|\leq1.
\]
Taking real parts on the unit circle gives
\[
 \bigl|\log|q_m(e^{it})|-u(t)\bigr|
 \le \frac{CM_H^2}{m},
 \qquad 0\le t\le2\pi.
\]
Choose $m$ so large that $CM_H^2/m\le a_0/4$. By
\eqref{eq:u-below-v},
\[
 \log|q_m(e^{it})|
 \le v(t)-\frac{a_0}{2},
 \qquad 0<t<2\pi.
\]
Therefore
\begin{equation}\label{eq:qm-weight}
 \Omega_L(t)|q_m(e^{it})|
 \le \Gamma e^{-a_0/2}<\Gamma,
 \qquad 0\le t\le2\pi,
\end{equation}
where the inequality at the endpoints is automatic because $\Omega_L(0)=\Omega_L(2\pi)=0$.

\smallskip\noindent\emph{Step 3: Restoration of monicity.} To obtain a monic polynomial define the reversed polynomial
\[
 {q_m^*(z)=z^{\deg q_m}\overline{q_m(1/\overline z)}}.
\]
Since $q_m(0)=1$, the leading coefficient of $q_m^*$ is $1$, so $q_m^*$ is
monic. Moreover, for $|z|=1$,
\[
 {q_m^*(z)=z^{\deg q_m}\overline{q_m(z)}},
 \qquad
 |q_m^*(z)|=|q_m(z)|.
\]
Thus
\[
 {Q(z)=(z-1)^Lq_m^*(z)}
\]
is monic, is divisible by $(z-1)^L$, and, by \eqref{eq:qm-weight}, satisfies $C(Q)<\Gamma$. Since $\Gamma>e^{-1/2}$ was arbitrary, the reverse inequality follows.

\end{proof}
Together with Proposition~\ref{prop:log-law}, Theorem~\ref{thm:Q}
therefore shows that the continuum extremal problem considered above
has sharp value $e^{-1/2}$.

\subsection{Discretization and the sharp coefficient asymptotics}\label{sec:eigenvalue-distribution}
We now discretize the optimal density to construct the zeros of the finite Blaschke product $B_n$. We also fix
\[
 r_n=1-n^{-1/6}.
\]
After the change of variable $u=t/(2\pi)$, this density is $-\log u$. Set
\[
 M(u)=u(1-\log u),\qquad 0<u\le1,\qquad M(0)=0.
\]
Then $M(u)=\nu([0,2\pi u])$ is the distribution function of $\nu$. Since $M'(u)=-\log u>0$ on $(0,1)$, $M$ is a strictly increasing bijection from $[0,1]$ onto $[0,1]$. Define
\[
 q_{j,n}=M^{-1}(j/n),\qquad 0\le j\le n,
\]
and
\[
 I_{j,n}=[2\pi q_{j-1,n},2\pi q_{j,n}].
\]
Then $\nu(I_{j,n})=1/n$. We choose
\[
 \theta_{j,n}=2\pi M^{-1}\!\left(\frac{j-\frac12}{n}\right),
 \qquad
 \lambda_{j,n}=r_ne^{i\theta_{j,n}},
 \qquad 1\le j\le n.
\]
The point $\theta_{j,n}$ is the image under $2\pi M^{-1}$ of the midpoint of $[(j-1)/n,j/n]$, and hence
\[
 2\pi q_{j-1,n}<\theta_{j,n}<2\pi q_{j,n}.
\]
Thus the angles, and therefore the points $\lambda_{j,n}$, are distinct.
Since $\theta_{j,n}\in I_{j,n}$ and the uniform continuity of $M^{-1}$ gives
\[
 \max_{1\leq j\leq n}|I_{j,n}|\to0,
\]
the empirical measures $\mu_n$ converge weakly to $\nu$; that is,
\[
 \int_0^{2\pi} g(t)\,d\mu_n(t)\longrightarrow
 \int_0^{2\pi} g(t)\,d\nu(t)
\]
for every continuous $2\pi$-periodic function $g$.
Set
\[
 \delta_n=n^{-1/32}, \qquad J_n=[\delta_n,2\pi-\delta_n].
\]
For $r_n=1-n^{-1/6}$ we prove in Appendix~\ref{sec:coeff-proof} that
\[
 \sup_{t\in J_n}
 \bigl|\tau_n^{(j)}(t)-\tau^{(j)}(t)\bigr|\to0,
 \qquad j=0,1,2.
\]

The following proposition identifies the asymptotic size of the largest
Taylor coefficient of $QB_n$.

\begin{proposition}\label{prop:coeff-norm}
Let $L=16$, and let $Q\not\equiv0$ be fixed with $(z-1)^L\mid Q$. Then
\begin{equation}\label{eq:coeff-norm}
 \sqrt n\,\|QB_n\|_{{l_\infty^A}}\longrightarrow C(Q).
\end{equation}
\end{proposition}
The proof is given in Appendix~\ref{sec:coeff-proof}. With
$\delta_n=n^{-1/32}$, negligibility of the endpoint contribution
requires
\[
 \sqrt n\,\delta_n^{L+1}\to0,
\]
which is equivalent to $L>15$. Thus $L=16$ is the smallest admissible
integer for this choice of $\delta_n$.
\begin{corollary}\label{cor:Qeta}
For every $\eta>0$, there is a fixed monic polynomial $Q_\eta$, divisible by
$(z-1)^{16}$, such that
\[
 \|Q_\eta B_n\|_{{l_\infty^A}}
 \le \frac{1+\eta}{\sqrt{en}}
\]
for all sufficiently large $n$.
\end{corollary}

\begin{proof}
Choose $Q_\eta$ from Theorem~\ref{thm:Q} so that
$C(Q_\eta)<(1+\eta/2)e^{-1/2}$, and apply
Proposition~\ref{prop:coeff-norm}.
\end{proof}

\subsection{Explicit matrices and completion of the proof}
\label{sec:explicit-matrices}

We now pass from the prescribed spectra constructed above to explicit
operators. The model-operator realization recalled in
Section~\ref{sec:duality} converts the upper bound of
Corollary~\ref{cor:Qeta} into a sequence of matrices asymptotically
attaining Sch\"affer's upper bound.

\begin{theorem}\label{thm:explicit-matrices}
For $n\ge2$, put $w_{j,n}=\overline{\lambda_{j,n}}$ and define the Malmquist--Walsh functions
\[
 e_{j,n}(z)=
 \frac{(1-|w_{j,n}|^2)^{1/2}}{1-\overline{w_{j,n}}z}
 \prod_{\ell=1}^{j-1}\frac{z-w_{\ell,n}}{1-\overline{w_{\ell,n}}z},
 \qquad 1\le j\le n,
\]
where the empty product is $1$. Equip $\C^n$ with the norm
\[
 \|x\|_*=
 \bigl\|\sum_{j=1}^n x_je_{j,n}\bigr\|_{{l_\infty^A}},
\]
and let $\|\cdot\|_*$ also denote the induced operator norm. Define the upper-triangular matrix
\[
 (T_n)_{ij}=
 \begin{cases}
 (1-r_n^2)\displaystyle\prod_{\ell=i+1}^{j-1}
        (-\overline{\lambda_{\ell,n}}),& i<j,\\[0.7em]
 \lambda_{i,n},& i=j,\\
 0,& i>j.
 \end{cases}
\]
Then the eigenvalues of $T_n$ are exactly the distinct points
$\lambda_{1,n},\ldots,\lambda_{n,n}$, one has $\|T_n\|_*\le1$, and
\[
 \frac{|\det T_n|\,\|T_n^{-1}\|_*}
      {\|T_n\|_*^{n-1}\sqrt n}
 \longrightarrow \sqrt e.
\]
\end{theorem}

\begin{proof}[Proof of Theorem~\ref{thm:explicit-matrices} and Theorem~\ref{thm:main-intro}]
\smallskip\noindent\emph{Step 1: The prescribed spectra.}
Fix $\eta>0$ and choose $Q_\eta$ as in Corollary~\ref{cor:Qeta}. Let
$d_n=\prod_{j=1}^n\lambda_{j,n}$. Since all eigenvalues have modulus $r_n$,
\[
 |d_n|=r_n^n=(1-n^{-1/6})^n\le e^{-n^{5/6}}.
\]
By Proposition~\ref{prop:holder-lower},
\[
 \varphi(\lambda_{1,n},\ldots,\lambda_{n,n})
 \ge \frac1{\|Q_\eta B_n\|_{{l_\infty^A}}}-|d_n|
 \ge \frac{\sqrt{en}}{1+\eta}-e^{-n^{5/6}}.
\]
Letting first $n\to\infty$ and then $\eta\to0$ gives
\[
 \liminf_{n\to\infty}
 \frac{\varphi(\lambda_{1,n},\ldots,\lambda_{n,n})}{\sqrt n}
 \ge\sqrt e.
\]

\smallskip\noindent\emph{Step 2: Model-operator realization.}
Coefficient conjugation is an isometry of $\W$, so
\[
 \varphi(w_{1,n},\ldots,w_{n,n})
 =\varphi(\lambda_{1,n},\ldots,\lambda_{n,n}).
\]
By \cite[proof of Theorem~8]{SZJMPA}, $T_n$ is the matrix of the backward shift in the Malmquist--Walsh basis above, and therefore $\|T_n\|_*\le1$. Moreover, \cite[Remark~12]{SZJMPA} gives
\[
 |\det T_n|\,\|T_n^{-1}\|_*
 =\varphi(w_{1,n},\ldots,w_{n,n})
 =\varphi(\lambda_{1,n},\ldots,\lambda_{n,n}).
\]
Hence
\[
 \liminf_{n\to\infty}
 \frac{|\det T_n|\,\|T_n^{-1}\|_*}
      {\|T_n\|_*^{n-1}\sqrt n}
 \ge\sqrt e,
\]
which, together with Sch\"affer's upper bound, completes the proof.
\end{proof}

\medskip
\noindent\textbf{AI statement.}
ChatGPT (OpenAI) was used to assist with the editing and writing of this
article. The research presented here forms part of a long-standing
research program of the authors, and all mathematical results and
arguments were developed by the authors.

\newpage
\appendix
\section{Appendix: Proof of Proposition~\ref{prop:coeff-norm}}\label{sec:coeff-proof}

Lemma~\ref{lem:phase-control} first gives uniform control of the phase and
its derivatives; the proof then uses stationary phase near the critical
point and van der Corput estimates elsewhere.

\subsection{Uniform control of the phase}

We work on
\[
 J_n=[\delta_n,2\pi-\delta_n],\qquad \delta_n=n^{-1/32}.
\]
The required $C^2$ control of $\tau_n$ on $J_n$ is established first.

\begin{lemma}\label{lem:phase-control}
With $r_n=1-n^{-1/6}$ and $\delta_n=n^{-1/32}$, one has, for
$j=0,1,2$,
\begin{equation}\label{eq:C2-phase}
 \sup_{t\in J_n}|\tau_n^{(j)}(t)-\tau^{(j)}(t)|\longrightarrow0.
\end{equation}
Consequently, for all sufficiently large $n$ and all $t\in J_n$,
\begin{equation}\label{eq:regularity}
 \frac1{2t}\le-\tau_n'(t)\le\frac2t,
 \qquad
 |\tau_n''(t)|\le\frac2{t^2}.
\end{equation}
In particular, $\tau_n$ is strictly decreasing on $J_n$.
\end{lemma}

\begin{proof}

Fix $j\in\{0,1,2\}$. We start from the difference in
\eqref{eq:C2-phase}. Since $d\nu(s)=\tau(s)\,ds/(2\pi)$ and
\eqref{eq:tau-poisson-measure} gives $\tau_n=P_{r_n}*\mu_n$, we have
\[
 \tau_n^{(j)}-\tau^{(j)}
 =
 \bigl[P_{r_n}*(\mu_n-\nu)\bigr]^{(j)}
 +
 \bigl[(P_{r_n}*\nu)^{(j)}-\tau^{(j)}\bigr].
\]
Here
\[
 (P_R*f)(t)=\frac1{2\pi}\int_0^{2\pi}P_R(t-s)f(s)\,ds,
 \qquad
 (P_R*\mu)(t)=\int_0^{2\pi}P_R(t-s)\,d\mu(s).
\]
We estimate the two terms separately. First consider the discretization error. The intervals
\[
 I_{\ell,n}=[2\pi q_{\ell-1,n},2\pi q_{\ell,n}],
 \qquad 1\le\ell\le n,
\]
have $\nu$-mass $1/n$, and $\theta_{\ell,n}\in I_{\ell,n}$. Hence, for every
$2\pi$-periodic $g\in C^1$,
\[
 \Bigl|\int_0^{2\pi} g\,d\mu_n-\int_0^{2\pi} g\,d\nu\Bigr|
 \le
 \sum_{\ell=1}^n\int_{I_{\ell,n}}
 |g(\theta_{\ell,n})-g(t)|\,d\nu(t)
 \le
 \frac{2\pi}{n}\|g'\|_\infty.
\]

Let $h=1-R$. For $R\in[1/2,1)$, direct differentiation of the Poisson
kernel, together with
\[
 1-2R\cos x+R^2=h^2+2R(1-\cos x)
\]
and the elementary bounds $\frac{2}{\pi^2}x^2\le 1-\cos x\le\frac12x^2$ for $|x|\le\pi$, shows that there is an absolute constant $C$ such that, for $0\le m\le3$,
\[
 \|P_R^{(m)}\|_\infty\le C h^{-m-1},
 \qquad
 |P_R^{(m)}(x)|\le C h|x|^{-m-2}
 \quad(2h\le |x|\le\pi).
\]
Taking $g(s)=P_{r_n}^{(j)}(t-s)$ and using
$1-r_n=n^{-1/6}$ gives

\begin{equation}\label{eq:discretization-error}
 \sup_t
 \bigl|\bigl[P_{r_n}*(\mu_n-\nu)\bigr]^{(j)}(t)\bigr|
 \le
 \frac{C}{n}(1-r_n)^{-j-2}
 =
 Cn^{(j-4)/6}
 \longrightarrow0.
\end{equation}

We next estimate the Poisson-smoothing error. Let $R\in[1/2,1)$,
$h=1-R$, $0<\delta\le1$, and $h\le\delta/8$. The preceding bounds imply
\[
 \frac1{2\pi}\int_{-\pi}^{\pi}|x|P_R(x)\,dx
 \le C h\log\frac eh.
\]
Let $f$ be the $2\pi$-periodic $L^1$ function which equals
$\log(2\pi/t)$ on $(0,2\pi)$. Choose a smooth $2\pi$-periodic function
$\chi_\delta$ which is $0$ when
$\operatorname{dist}(t,2\pi\mathbb Z)\le\delta/4$, is $1$ when
$\operatorname{dist}(t,2\pi\mathbb Z)\ge\delta/2$, and satisfies
\[
 \|\chi_\delta^{(m)}\|_\infty\le C\delta^{-m},
 \qquad 0\le m\le3.
\]
Set $f_0=\chi_\delta f$ and $f_1=(1-\chi_\delta)f$. For
$t\in[\delta,2\pi-\delta]$, one has
$f_0^{(j)}(t)=f^{(j)}(t)$ and
\[
 (P_R*f)^{(j)}(t)-f^{(j)}(t)
 =
 P_R*(f_0^{(j)})(t)-f_0^{(j)}(t)
 +
 P_R^{(j)}*f_1(t).
\]
Since $j\le2$, Leibniz' rule and the explicit derivatives of $f$ give
\[
 \|f_0^{(j+1)}\|_\infty
 \le
 C\delta^{-j-1}\left(1+\log\frac1\delta\right),
 \qquad
 \|f_1\|_{L^1(0,2\pi)}
 \le
 C\delta\left(1+\log\frac1\delta\right).
\]
The first-moment estimate for $P_R$ yields
\[
 |P_R*(f_0^{(j)})(t)-f_0^{(j)}(t)|
 \le
 C h\delta^{-j-1}
 \left(1+\log\frac1\delta\right)\log\frac eh.
\]
Moreover, the support of $f_1$ is contained in the
$\delta/2$-neighborhood of $2\pi\mathbb Z$. Thus its circular distance
from $t\in[\delta,2\pi-\delta]$ is at least $\delta/2\ge4h$.
Using the pointwise bound for $P_R^{(j)}$ above,
\[
 |P_R^{(j)}*f_1(t)|
 \le
 C h\delta^{-j-1}
 \left(1+\log\frac1\delta\right).
\]
Therefore
\[
 \sup_{\delta\le t\le2\pi-\delta}
 |(P_R*f)^{(j)}(t)-f^{(j)}(t)|
 \le
 C h\delta^{-j-1}
 \left(1+\log\frac1\delta\right)\log\frac eh.
\]

We now take $R=r_n$, so that $h=n^{-1/6}$, and
$\delta=\delta_n=n^{-1/32}$. 
Since $d\nu=f\,dt/(2\pi)$, we have $P_{r_n}*f=P_{r_n}*\nu$. For all sufficiently large $n$, $h\le\delta_n/8$, and
\[
 \sup_{t\in J_n}
 |(P_{r_n}*\nu)^{(j)}(t)-\tau^{(j)}(t)|
 \le
 C n^{-1/6+(j+1)/32}(\log n)^2
 \longrightarrow0,
\]
since $j\leq2$. Together with \eqref{eq:discretization-error} this proves
\eqref{eq:C2-phase}. Finally
\[
 \tau'(t)=-\frac1t,
 \qquad
 \tau''(t)=\frac1{t^2}.
\]
By \eqref{eq:C2-phase}, for all sufficiently large $n$,
\[
 \sup_{J_n}|\tau_n'-\tau'|\le\frac1{4\pi},
 \qquad
 \sup_{J_n}|\tau_n''-\tau''|\le\frac1{4\pi^2}.
\]
Since $t\le2\pi$ on $J_n$,
\[
 -\tau_n'(t)
 \ge\frac1t-\frac1{4\pi}
 \ge\frac1{2t},
 \qquad
 -\tau_n'(t)
 \le\frac1t+\frac1{4\pi}
 \le\frac2t,
\]
and
\[
 |\tau_n''(t)|
 \le\frac1{t^2}+\frac1{4\pi^2}
 \le\frac2{t^2}.
\]
This proves \eqref{eq:regularity} and the strict decrease of $\tau_n$ on
$J_n$.

\end{proof}

\subsection{The oscillatory integral}

\begin{proof}[Proof of Proposition~\ref{prop:coeff-norm}]
{
For $k\ge0$, put
\[
 A_Q(t)=Q(e^{it}),
 \qquad
 a=\frac{k}{n},
 \qquad
 \psi_{n,a}(t)=\frac{\phi_n(t)}n-at.
\]
Then
\[
 \widehat{QB_n}(k)
 =\frac1{2\pi}\int_0^{2\pi}A_Q(t)e^{in\psi_{n,a}(t)}\,dt,
\]
and, by \eqref{eq:phase-derivatives-motivation},
\[
 \psi_{n,a}'(t)=\tau_n(t)-a,
 \qquad
 \psi_{n,a}''(t)=\tau_n'(t).
\]

We first separate the endpoint region from the interior interval:
\[
 E_n=[0,\delta_n]\cup[2\pi-\delta_n,2\pi],
 \qquad
 J_n=[\delta_n,2\pi-\delta_n].
\]
We begin by estimating the contribution of $E_n$, uniformly in $k$. Since $Q$ is divisible by
$(z-1)^L$, there is a constant $C_Q>0$, depending only on $Q$, such that,
for every $0<\varepsilon<1$,
\begin{equation}\label{eq:amplitude-ends}
\begin{aligned}
 &\sup_{0\le t\le4\varepsilon}|A_Q(t)|
 +\int_0^{4\varepsilon}|A_Q'(t)|\,dt\\
 &\quad+
 \sup_{2\pi-4\varepsilon\le t\le2\pi}|A_Q(t)|
 +\int_{2\pi-4\varepsilon}^{2\pi}|A_Q'(t)|\,dt
 \le C_Q\varepsilon^L,
\end{aligned}
\end{equation}
and
\[
 \|A_Q\|_\infty+\int_0^{2\pi}|A_Q'(t)|\,dt\le C_Q.
\]
In particular, since $|e^{in\psi_{n,a}(t)}|=1$,
\begin{equation}\label{eq:endpoints}
 \sup_{k\ge0}\sqrt n\,\Bigl|
 \frac1{2\pi}\int_{E_n}A_Q(t)e^{in\psi_{n,a}(t)}\,dt\Bigr|
 \le C_Q\sqrt n\,\delta_n^{L+1}
 =C_Qn^{-1/32}.
\end{equation}
Thus it remains to study the integral over $J_n$.

We shall use only the following two forms of van der Corput's lemma
\cite[Chapter~VIII, \S1.2, Corollary, p.~334]{Stein}. There is an absolute
constant $C$ such that, on any interval $I$,
\begin{equation}\label{eq:vdc1}
 \Bigl|\int_I g(t)e^{in\phi(t)}\,dt\Bigr|
 \le\frac{C}{n\gamma}
 \left(\|g\|_\infty+\int_I|g'(t)|\,dt\right)
\end{equation}
if $\phi'$ is monotone and $|\phi'|\ge\gamma$ on $I$, while
\begin{equation}\label{eq:vdc2}
 \Bigl|\int_I g(t)e^{in\phi(t)}\,dt\Bigr|
 \le\frac{C}{\sqrt{n\gamma}}
 \left(\|g\|_\infty+\int_I|g'(t)|\,dt\right)
\end{equation}
if $|\phi''|\ge\gamma$ on $I$. In what follows, we use the standard Landau notation $O_\varepsilon(\cdot)$, $O_Q(\cdot)$, and $O_{Q,\varepsilon}(\cdot)$: the constants may depend on $\varepsilon$, on $Q$, or on both, but not on $n$, $k$, or $a$. We also write $X\asymp_\varepsilon Y$ when $c_\varepsilon Y\le X\le C_\varepsilon Y$ for some positive constants $c_\varepsilon,C_\varepsilon$ depending only on $\varepsilon$. All these bounds are uniform in $n$, $k$, and $a$.

Fix $\varepsilon\in(0,\pi/8)$ and take $n$ so large that
$\delta_n<2\varepsilon$. By Lemma~\ref{lem:phase-control}, $\tau_n$ is
strictly decreasing on $J_n$. Set
\[
 \alpha_n=\tau_n(2\pi-\delta_n),
 \qquad
 \alpha_{n,\varepsilon}=\tau_n(2\pi-2\varepsilon),
\]
\[
 \beta_{n,\varepsilon}=\tau_n(2\varepsilon),
 \qquad
 \beta_n=\tau_n(\delta_n).
\]
Then
\[
 \alpha_n<\alpha_{n,\varepsilon}
 <\beta_{n,\varepsilon}<\beta_n.
\]
We divide the possible values of $a=k/n$ into five regions:
\begin{enumerate}
\item[(I)] $a<\alpha_n$: there is no stationary point in $J_n$;
\item[(II)] $\alpha_n\le a<\alpha_{n,\varepsilon}$: the unique stationary
point belongs to $(2\pi-2\varepsilon,2\pi-\delta_n]$;
\item[(III)] $\alpha_{n,\varepsilon}\le a\le\beta_{n,\varepsilon}$: the
unique stationary point belongs to $[2\varepsilon,2\pi-2\varepsilon]$;
\item[(IV)] $\beta_{n,\varepsilon}<a\le\beta_n$: the unique stationary
point belongs to $[\delta_n,2\varepsilon)$;
\item[(V)] $a>\beta_n$: there is no stationary point in $J_n$.
\end{enumerate}
These regions are exhaustive because $\psi_{n,a}'=\tau_n-a$ and $\tau_n$ is
strictly decreasing. Cases (I)--(II) concern the right endpoint and Cases
(IV)--(V) the left endpoint. Case (III) contains an interior stationary point
and is the only one that produces the leading term. In the other four cases,
van der Corput's lemma, together with the zero of $Q$ at $1$, gives a
negligible contribution.

\medskip\noindent\emph{Cases (I)--(II): the right endpoint.}
In both cases split $J_n$ at $2\pi-4\varepsilon$. On
$[2\pi-4\varepsilon,2\pi-\delta_n]$, Lemma~\ref{lem:phase-control} gives
\[
 |\psi_{n,a}''(t)|=-\tau_n'(t)\ge\frac1{2t}\ge\frac1{4\pi}.
\]
Hence \eqref{eq:vdc2} and \eqref{eq:amplitude-ends} show that this part of
the integral is $O_Q(\varepsilon^L n^{-1/2})$.

On $[\delta_n,2\pi-4\varepsilon]$ we use \eqref{eq:vdc1}. In Case~(I),
for $t\le2\pi-4\varepsilon$,
\begin{align*}
 \psi_{n,a}'(t)
 &\ge \tau_n(2\pi-4\varepsilon)-\tau_n(2\pi-\delta_n)\\
 &=\int_{2\pi-4\varepsilon}^{2\pi-\delta_n}(-\tau_n'(s))\,ds
 \ge\frac{\varepsilon}{2\pi}.
\end{align*}
In Case~(II), if $t_{n,a}$ is the stationary point, then
$t_{n,a}>2\pi-2\varepsilon$, and therefore
\begin{align*}
 \psi_{n,a}'(t)
 &\ge \tau_n(2\pi-4\varepsilon)-\tau_n(2\pi-2\varepsilon)\\
 &=\int_{2\pi-4\varepsilon}^{2\pi-2\varepsilon}(-\tau_n'(s))\,ds
 \ge\frac{\varepsilon}{2\pi}.
\end{align*}
Since $\psi_{n,a}'$ is decreasing, \eqref{eq:vdc1} gives
$O_{Q,\varepsilon}(n^{-1})$ for this part. Together with
\eqref{eq:endpoints}, uniformly in Cases~(I)--(II),
\begin{equation}\label{eq:right-cases}
 \sqrt n\,|\widehat{QB_n}(k)|
 =O_Q(\varepsilon^L+n^{-1/32})+O_{Q,\varepsilon}(n^{-1/2}).
\end{equation}
The boundary value $a=\alpha_n$ is included in Case~(II).

\medskip\noindent\emph{Case (III): an interior stationary point.}
Let $t_{n,a}\in[2\varepsilon,2\pi-2\varepsilon]$ be the unique solution of
$\tau_n(t)=a$, and put
\[
 t_a=2\pi e^{-a},
\]
so that $\tau(t_a)=a$.

\smallskip\noindent\emph{Location of the stationary point.}
Since $\tau_n(t_{n,a})=\tau(t_a)$, Lemma~\ref{lem:phase-control} gives,
uniformly in Case~(III),
\[
 \bigl|\log\frac{t_{n,a}}{t_a}\bigr|
 =|\tau_n(t_{n,a})-\tau(t_{n,a})|
 \le\sup_{t\in J_n}|\tau_n(t)-\tau(t)|\longrightarrow0,
\]
and
\[
 \bigl|-\tau_n'(t_{n,a})-\frac1{t_{n,a}}\bigr|
 \le\sup_{t\in J_n}|\tau_n'(t)-\tau'(t)|\longrightarrow0.
\]
Hence
\begin{equation}\label{eq:critical-limit}
 \sup_{\substack{k\ge0\\
  \alpha_{n,\varepsilon}\le k/n\le\beta_{n,\varepsilon}}}
 |t_{n,k/n}-t_{k/n}|\longrightarrow0,
 \qquad
 \sup_{\substack{k\ge0\\
  \alpha_{n,\varepsilon}\le k/n\le\beta_{n,\varepsilon}}}
 \bigl|t_{n,k/n}[-\tau_n'(t_{n,k/n})]-1\bigr|\longrightarrow0.
\end{equation}

\smallskip\noindent\emph{Local stationary-phase estimate.}
On the fixed interval $[\varepsilon,2\pi-\varepsilon]$ there are constants
$c_\varepsilon,C_\varepsilon>0$ such that, for all large $n$,
\begin{equation}\label{eq:central-derivatives}
 c_\varepsilon\le-\psi_{n,a}''(t)\le C_\varepsilon,
 \qquad
 |\psi_{n,a}^{(3)}(t)|\le C_\varepsilon
\end{equation}
for every $a$ in Case~(III).

Set $\rho_n=n^{-2/5}$ and $t_0=t_{n,a}$. For large $n$,
$[t_0-\rho_n,t_0+\rho_n]\subset[\varepsilon,2\pi-\varepsilon]$. Write
\[
 \frac1{2\pi}\int_{J_n}A_Q(t)e^{in\psi_{n,a}(t)}\,dt
 =I^-_{n,a}+I^0_{n,a}+I^+_{n,a},
\]
where $I^0_{n,a}$ is the integral over
$[t_0-\rho_n,t_0+\rho_n]$ and $I^-_{n,a},I^+_{n,a}$ are the two remaining
parts of $J_n$.

Fedoryuk's one-dimensional stationary-phase theorem
\cite[Chapter~III, \S1, Theorem~1.7]{Fedoryuk77},
applied to the phase $-n\psi_{n,a}$ and then conjugating, gives the same
leading term for
\[
 \int_{t_0-\rho_n}^{t_0+\rho_n}
 e^{in\psi_{n,a}(t)}\,dt.
\]
We give below a direct proof for $I^0_{n,a}$, which also includes the
factor $A_Q(t)$, with a remainder $O_{Q,\varepsilon}(n^{-3/5})$,
uniform in $a$, and hence $o(n^{-1/2})$. Put
\[
 \Psi_{n,a}(t)=-n\psi_{n,a}(t),
 \qquad
 \kappa_{n,a}=\Psi_{n,a}''(t_0)
 =n[-\tau_n'(t_0)].
\]
By \eqref{eq:central-derivatives},
$\kappa_{n,a}\asymp_\varepsilon n$. With
$\ell_{n,a}=\rho_n\sqrt{\kappa_{n,a}}$, one has
$\ell_{n,a}\asymp_\varepsilon n^{1/10}\to\infty$. Therefore the Fresnel
integral and one integration by parts in its tails give
\begin{equation}\label{eq:quadratic-uniform}
 \int_{-\rho_n}^{\rho_n}e^{i\kappa_{n,a}x^2/2}\,dx
 =\sqrt{\frac{2\pi}{\kappa_{n,a}}}\,e^{i\pi/4}
 +O_\varepsilon\!\left(\frac1{\kappa_{n,a}\rho_n}\right)
 =\sqrt{\frac{2\pi}{\kappa_{n,a}}}\,e^{i\pi/4}
 +O_\varepsilon(n^{-3/5}).
\end{equation}
Moreover, Taylor's formula and \eqref{eq:central-derivatives} yield, for
$|x|\le\rho_n$,
\[
 \bigl|\Psi_{n,a}(t_0+x)-\Psi_{n,a}(t_0)
 -\frac{\kappa_{n,a}}2x^2\bigr|
 \le C_\varepsilon n|x|^3.
\]
Since $|e^{iu}-e^{iv}|\le|u-v|$, integrating this bound gives an error
$O_\varepsilon(n\rho_n^4)=O_\varepsilon(n^{-3/5})$. Combining it with
\eqref{eq:quadratic-uniform}, and then returning to $\psi_{n,a}$, gives
\[
 \int_{t_0-\rho_n}^{t_0+\rho_n}e^{in\psi_{n,a}(t)}\,dt
 =\sqrt{\frac{2\pi}{n[-\psi_{n,a}''(t_0)]}}
 e^{i(n\psi_{n,a}(t_0)-\pi/4)}
 +O_\varepsilon(n^{-3/5}).
\]
Finally,
\[
 |A_Q(t)-A_Q(t_0)|\le\|A_Q'\|_\infty|t-t_0|,
\]
so replacing $A_Q(t)$ by $A_Q(t_0)$ in $I^0_{n,a}$ costs
$O_Q(\rho_n^2)=O_Q(n^{-4/5})$. Thus
\begin{equation}\label{eq:central}
 I^0_{n,a}
 =\frac{e^{i(n\psi_{n,a}(t_0)-\pi/4)}A_Q(t_0)}
 {\sqrt{2\pi n(-\tau_n'(t_0))}}
 +O_{Q,\varepsilon}(n^{-3/5}).
\end{equation}

It remains to estimate $I^-_{n,a}$ and $I^+_{n,a}$. Since $\psi_{n,a}'(t_0)=0$, \eqref{eq:central-derivatives} gives
\[
 \psi_{n,a}'(t_0-\rho_n)\ge c_\varepsilon\rho_n,
 \qquad
 -\psi_{n,a}'(t_0+\rho_n)\ge c_\varepsilon\rho_n.
\]
Because $\psi_{n,a}'$ is decreasing, the same lower bound for its modulus
holds on the whole domains of $I^-_{n,a}$ and $I^+_{n,a}$. Hence
\eqref{eq:vdc1} gives
\begin{equation}\label{eq:sides}
 |I^-_{n,a}|+|I^+_{n,a}|
 =O_{Q,\varepsilon}\!\left(\frac1{n\rho_n}\right)
 =O_{Q,\varepsilon}(n^{-3/5}).
\end{equation}
After multiplication by $\sqrt n$, the main term in \eqref{eq:central} has
modulus
\[
 \frac{|A_Q(t_0)|}{\sqrt{2\pi[-\tau_n'(t_0)]}},
\]
which converges uniformly, by \eqref{eq:critical-limit}, to
$\sqrt{t_a/(2\pi)}|Q(e^{it_a})|$. Combining this with
\eqref{eq:endpoints} and \eqref{eq:sides}, we obtain
\begin{equation}\label{eq:central-main}
 \sup_{\substack{k\ge0\\
  \alpha_{n,\varepsilon}\le k/n\le\beta_{n,\varepsilon}}}
 \Bigl|\sqrt n\,|\widehat{QB_n}(k)|
 -\sqrt{\frac{t_{k/n}}{2\pi}}|Q(e^{it_{k/n}})|\Bigr|
 \longrightarrow0.
\end{equation}

\medskip\noindent\emph{Cases (IV)--(V): the left endpoint.}
In both cases split $J_n$ at $4\varepsilon$. On
$[\delta_n,4\varepsilon]$,
\[
 |\psi_{n,a}''(t)|=-\tau_n'(t)\ge\frac1{2t}\ge\frac1{8\varepsilon},
\]
so \eqref{eq:vdc2} and \eqref{eq:amplitude-ends} give
$O_Q(\varepsilon^{L+1/2}n^{-1/2})$ for this part.

On $[4\varepsilon,2\pi-\delta_n]$ we use \eqref{eq:vdc1}. In Case~(IV),
if $t_{n,a}<2\varepsilon$ is the stationary point, then for
$t\ge4\varepsilon$,
\[
 -\psi_{n,a}'(t)
 \ge\tau_n(2\varepsilon)-\tau_n(4\varepsilon)
 =\int_{2\varepsilon}^{4\varepsilon}(-\tau_n'(s))\,ds
 \ge\frac12\log2.
\]
In Case~(V),
\[
 |\psi_{n,a}'(t)|
 \ge\tau_n(\delta_n)-\tau_n(4\varepsilon)
 \ge\int_{2\varepsilon}^{4\varepsilon}\frac{ds}{2s}
 =\frac12\log2.
\]
Thus the second part is $O_Q(n^{-1})$. Together with
\eqref{eq:endpoints}, uniformly in Cases~(IV)--(V),
\begin{equation}\label{eq:left-cases}
 \sqrt n\,|\widehat{QB_n}(k)|
 =O_Q(\varepsilon^{L+1/2}+n^{-1/2}+n^{-1/32}).
\end{equation}
The boundary value $a=\beta_n$ is included in Case~(IV).

We now take the supremum over $k$. For fixed $\varepsilon$,
\eqref{eq:right-cases}, \eqref{eq:central-main}, and
\eqref{eq:left-cases} show that there is a constant $K_Q>0$, independent
of $\varepsilon$, such that
\[
 \limsup_{n\to\infty}\sqrt n\,\|QB_n\|_{l_\infty^A}
 \le \max\{C(Q),K_Q\varepsilon^L\},
\]
because $0<\varepsilon<1$. Letting $\varepsilon\to0$ yields
\begin{equation}\label{eq:coeff-upper}
 \limsup_{n\to\infty}\sqrt n\,\|QB_n\|_{l_\infty^A}\le C(Q).
\end{equation}

For the reverse inequality, let
\[
 F(t)=\sqrt{\frac{t}{2\pi}}|Q(e^{it})|.
\]
The function $F$ vanishes at both endpoints and is not identically zero, so
it attains its maximum $C(Q)$ at some $t_*\in(0,2\pi)$. Put
\[
 a_*=\tau(t_*)=\log\frac{2\pi}{t_*},
\]
and let $k_n$ be the nearest integer to $na_*$. Choose
$\varepsilon\in(0,\pi/8)$ so small that
$2\varepsilon<t_*<2\pi-2\varepsilon$. Since
$k_n/n\to a_*$ and $\tau_n\to\tau$ uniformly at the two fixed points
$2\varepsilon$ and $2\pi-2\varepsilon$, for all sufficiently large $n$,
\[
 \alpha_{n,\varepsilon}<\frac{k_n}{n}<\beta_{n,\varepsilon}.
\]
Thus $k_n/n$ belongs to Case~(III). Since
$t_{k_n/n}=2\pi e^{-k_n/n}\to t_*$, \eqref{eq:central-main} gives
\[
 \sqrt n\,|\widehat{QB_n}(k_n)|\longrightarrow F(t_*)=C(Q).
\]
Therefore
\[
 \liminf_{n\to\infty}\sqrt n\,\|QB_n\|_{l_\infty^A}\ge C(Q).
\]
Together with \eqref{eq:coeff-upper}, this proves \eqref{eq:coeff-norm}.
}
\end{proof}

\end{document}